\documentclass[letterpaper, 10 pt, conference]{ieeeconf}
\IEEEoverridecommandlockouts
\usepackage{cite}
\usepackage{amsmath,amssymb,amsfonts}
\usepackage{algorithmic}
\usepackage{graphicx}
\usepackage{textcomp}
\usepackage{prettyref}

\usepackage{pgfpages}
\usepackage{tikz, pgfplots}
\usepackage{mathtools}

\newcommand\scalemath[2]{\scalebox{#1}{\mbox{\ensuremath{\displaystyle #2}}}} 
\newtheorem{theorem}{Theorem}[section]

\DeclareRobustCommand{\vec}[1]{ 				
	\ifthenelse{\equal{#1}{\omega} \OR \equal{#1}{\varphi} \OR \equal{#1}{\alpha} \OR \equal{#1}{\beta} \OR \equal{#1}{\chi} \OR \equal{#1}{\delta} \OR \equal{#1}{\varepsilon} \OR \equal{#1}{\phi} \OR \equal{#1}{\epsilon} \OR \equal{#1}{\gamma} \OR \equal{#1}{\eta} \OR \equal{#1}{\iota} \OR \equal{#1}{\kappa} \OR \equal{#1}{\lambda} \OR \equal{#1}{\mu} \OR \equal{#1}{\nu} \OR \equal{#1}{\pi} \OR \equal{#1}{\theta} \OR \equal{#1}{\vartheta} \OR \equal{#1}{\rho} \OR \equal{#1}{\sigma} \OR \equal{#1}{\varsigma} \OR \equal{#1}{\tau} \OR \equal{#1}{\upsilon} \OR \equal{#1}{\xi} \OR \equal{#1}{\psi} \OR \equal{#1}{\zeta}}{
		\boldsymbol{#1}
	}{
		\mathbf{#1}
	}
}
 
\newcommand{\dd}{\mathop{}\!\mathrm{d}}
\newcommand{\Diff}[2]{\frac{\dd#1}{\dd#2}}

\newcommand{\DDiff}[2]{\frac{\dd^2#1}{\dd#2^2}}

\newcommand{\Part}[1]{\partial_{#1}}

\newcommand{\PartDiffT}[1]{\Diff{#1}{t}}

\newcommand{\PartDDiffT}[1]{\DDiff{#1}{t}}

\newcommand{\lie}[1]{\mathrm{L}_{#1}}

\newcommand{\liebrack}[2]{\left[{#1},{#2}\right]}
\newcommand{\abs}[1]{\left\vert#1\right\vert}

\providecommand{\of}[1]{\left(#1\right)}

\DeclareMathOperator{\sign}{sgn} 

\newcommand{\diag}{\operatorname*{diag}}

\newcommand{\Span}[1]{\mathrm{span}\{#1\}}

\AtEndOfClass{}

\AtEndOfClass{
\newcommand{\ExtraTabEqnSpace}{1ex}
\makeatletter
\newcommand*{\TabEqn}[1]{%
\begingroup
	\setbox\@tempboxa=\hbox{%
	#1%
	}%
	\setlength{\dimen@}{\ht\@tempboxa}%
	\addtolength{\dimen@}{\ExtraTabEqnSpace}%
	\setlength{\ht\@tempboxa}{\dimen@}%
	\setlength{\dimen@}{\dp\@tempboxa}%
	\addtolength{\dimen@}{\ExtraTabEqnSpace}%
	\setlength{\dp\@tempboxa}{\dimen@}%
	\usebox\@tempboxa
\endgroup
}
\makeatother
}

\newrefformat{fig}{Figure~\ref{#1}}
\newrefformat{eq}{(\ref{#1})}
\newrefformat{subeq}{(\ref{#1})}
\newrefformat{apx}{Appendix~\ref{#1}}
\newrefformat{cha}{Chapter~\ref{#1}}
\newrefformat{sec}{Section~\ref{#1}}
\newrefformat{sub}{Section~\ref{#1}}
\newrefformat{tab}{Table~\ref{#1}}
\newrefformat{def}{Definition~\ref{#1}}
\newrefformat{th}{Theorem~\ref{#1}}

\begin{document}
\title{\textbf{Analytic Optimal Control for a Class of Driftless x-Flat Systems}}
\author{Raphael Buchinger, Georg Hartl, Lukas Ecker, Markus Schöberl
\thanks{This work has been submitted to the IEEE for possible publication. Copyright may be transferred without notice, after which this version may no longer be accessible.}
\thanks{This research was funded in whole, or in part, by
the Austrian Science Fund (FWF) P36473. For the purpose of
open access, the author has applied a CC BY public copyright
licence to any Author Accepted Manuscript version arising
from this submission.}
\thanks{Raphael Buchinger, Georg Hartl, Lukas Ecker and Markus Schöberl are with the Institute
of Control Systems, Johannes
Kepler University, 4040 Linz, Austria (e-mail: raphael.buchinger@jku.at; georg.hartl@jku.at; lukas.ecker@jku.at; markus.schoeberl@jku.at).}
}

\maketitle
 
\begin{abstract}
This paper studies optimal trajectory-tracking for driftless, x-flat nonlinear systems with three states and two inputs. The tracking problem is formulated in Bolza form with a quadratic cost of the tracking error and its derivative. Applying Pontryagin's maximum principle yields a mixed regular-singular optimal control problem. By exploiting geometric properties and a specific relation between the weighting matrices, a closed-form expression for the costate and an explicit feedback law for both inputs is derived. Thereby, the numerical solution of a two-point boundary-value problem is avoided. The singular input leads to a bang-singular-bang optimal control structure, while on the singular arc, the tracking error dynamics reduces to a linear dynamics of order two. The approach is illustrated for the kinematic model of a steerable axle, demonstrating accurate trajectory-tracking.
\end{abstract}

\section{Introduction}
\label{sec:introduction}
One of the primary objectives of control theory is that the trajectories of a nonlinear system \vspace{-1ex}
\begin{equation}\label{eq:f_xu_m_inputs}
    \dot x = f(x,u)\vspace{-1ex}
\end{equation}
with $n$ states and $m$ inputs follow desired reference trajectories, commonly referred to as the trajectory-tracking problem. In the field of trajectory-tracking, two typical approaches can be distinguished, namely optimization-based methods and flatness-based methods. The purpose of this work is to combine both viewpoints within one framework for a certain class of systems.

In optimal control, trajectory-tracking is usually formulated as an optimal tracking problem with a quadratic cost functional \cite{lewis2012optimal, lober_optimal_2017}. Two primary solution approaches are the Hamilton-Jacobi-Bellman equation and Pontryagin's maximum principle (PMP). This work uses only the latter. However, synthesizing an optimal trajectory-tracking controller is challenging, even for low-dimensional systems. This is because the necessary conditions for optimality lead to a two-point boundary-value problem (TPBVP) and solving the TPBVP numerically requires substantial computational effort, which generally limits real-time applicability \cite{bryson1975applied}. 

For systems that are differentially flat, a framework that avoids these computational limitations for trajectory-tracking exists. A nonlinear system of the form \prettyref{eq:f_xu_m_inputs} is called \emph{differentially flat} if there exists an $m$-tuple of functions $y = \varphi(x, u, \dot u, \ldots, u^{(\nu)})$, the so-called \emph{flat output}, such that all states and inputs admit a \emph{flat parameterization} $(x, u) = F(y, \dot y, \ldots, y^{(r)})$ in terms of $y$ and finitely many of its time derivatives \cite{fliess_flatness_1995}. Note that the superscripts in round brackets denote the order of the respective time derivative. If the system allows for flat outputs of the form $y=\varphi(x,u)$ or $y=\varphi(x)$, the system is called $(x,u)$-flat or $x$-flat, respectively. 
As shown in \cite{fliess_flatness_1995,fliess_lie-backlund_1999}, flat systems allow for systematic trajectory planning and tracking. Trajectory generation is treated as a separate task and the construction of such trajectories is not the focus of this work. Methods for optimal trajectory generation for differentially flat systems by the use of the flat parametrization can be found, for example, in \cite{rams_optimal_2018,beaver_optimal_2024, guay_real-time_2006}.
Flatness-based tracking control design typically relies on exact linearization by endogenous feedback\footnote{Informally, for a system of the form \eqref{eq:f_xu_m_inputs} an endogenous feedback only involves states $x$, inputs $u$ and time derivatives of the inputs.}, transforming the closed-loop system into the so-called Brunovsk\'{y} normal form, i.e. a system of $m$ decoupled integrator chains $y_j^{(l_j)} = v_j$ of respective lengths $l_j$ between a new input $v$ and the flat output $y$. The tracking problem is then solved by linear control design methods.
As shown in \cite{delaleau_control_1998}, every flat system can be exactly linearized by a quasi-static feedback law $u=\alpha(\tilde{x}, v, \dot{v}, \ldots, v^{(\Tilde{s})})$ that depends on suitably chosen time derivatives of the flat output, collected in a so-called generalized Brunovsk\'{y} state $\tilde{x}$. While this preserves the order of the closed-loop dynamics, obtaining the required derivatives of the flat output from measurement data is often difficult in practice. For $(x,u)$-flat systems this difficulty can be circumvented, as there is guaranteed to exist a quasi-static feedback law $u=\alpha(x,v,\dot v, \ldots, v^{(s)})$
that depends on the classical state $x$ rather than the generalized Brunovsk\'{y} state, see~\cite{delaleau_control_1998, gstottner_tracking_2024}. As a result, one obtains a tracking control law of the form\vspace{-1ex}
\begin{equation}\label{eq:tracking_law}
    u = \alpha(x,\, y_d(t), \dot y_d(t), \ldots, y_d^{(r)}(t)),\vspace{-1ex}
\end{equation}
which depends solely on the current state $x$ and the reference trajectory $y_d(t)$ along with finitely many of its time derivatives. 

In this work, differential flatness is not used in the classical sense of trajectory parametrization or exact feedback linearization. Instead, the flatness-based characterization of the system class under consideration provides the geometric structure that enables the analytic PMP-based tracking design developed below.
In contrast to flatness-based approaches, the proposed framework does not require a separate design of a state-feedback law and a trajectory-tracking controller, however the resultion control law is again of the form \prettyref{eq:tracking_law}. 

A key element of the derivation is a specific quadratic structure of the cost functional, which permits an analytic, closed-form solution to the optimal tracking problem rather than a numerical solution of the associated TPBVP. As a result, the need to numerically solve the TPBVP at each time step is avoided, thereby facilitating real-time implementation. Moreover, the analysis shows that along a singular arc, a linear tracking-error dynamics of lower order than the original system can be obtained. The main contributions are
\begin{itemize}
    \item the development of an analytic PMP-based trajectory-tracking control law for a $x$-flat driftless nonlinear system with three states and two inputs,
    \item the derivation of a closed-form solution to the associated TPBVP through an explicit representation of the costate,
    \item the proof that the resulting optimal control exhibits a \textit{bang-singular-bang} structure and that the tracking error dynamics on the singular arc is linear and second-order.
\end{itemize}

\section{Notation and Preliminaries}\label{sec:prelim}
This section recalls standard definitions and notation used throughout this work. Tensor notation and the Einstein summation convention are employed whenever the index range is clear from the context. Given a symmetric matrix $M = [M_{ij}] \in \mathbb{R}^{n\times n}$, we write $M\succ0$ ($M\succeq0$) to indicate that $M$ is positive (semi-)definite.
\subsection{Differential Geometry}\label{subsec:geometry}
Let $\mathcal{M}$ denote an $n$-dimensional state manifold equipped with local coordinates $x=(x^1,\ldots,x^n)$ and let $u=(u^1,\ldots,u^m)\in\mathcal{U}\subset\mathbb{R}^m$ denote the $m$-dimensional input. Furthermore, let $\mathcal{T}(\mathcal{M})$ and $\mathcal{T}^*(\mathcal{M})$ denote the tangent bundle and cotangent bundle of $\mathcal{M}$, respectively. For a smooth function $\varphi^i \in C^{\infty}(\mathcal{M})$ and a smooth vector field $g = g^k\partial_{x^k} \in \mathcal{T}(\mathcal{M})$, where $\partial_{x^k} := \frac{\partial}{\partial x^k}$, the Lie derivative of $\varphi^i$ along $g$ is given by $\lie{g}\varphi^i := g^k\partial_{x^k}\varphi^i$, and iterated Lie derivatives are denoted by $\lie{g}^l\varphi^i := \lie{g}(\lie{g}^{l-1}\varphi^i)$. The differential of $\varphi^i$ is written as $\mathrm{d}\varphi^i$, and for an $m$-tuple $\varphi = (\varphi^1,\ldots,\varphi^m)$ we write $\mathrm{d}\varphi = (\mathrm{d}\varphi^1, \ldots, \mathrm{d}\varphi^m)$.
Given two vector fields $g_1, g_2 \in \mathcal{T}(\mathcal{M})$, the Lie bracket is defined as $[g_1,g_2] = (g_1^l\partial_{x^l}g_2^k - g_2^l\partial_{x^l}g_1^k)\partial_{x^k}$.

\subsection{Optimal Control Problem}
In the following, we consider nonlinear affine input systems of the form
\begin{equation} \label{eq:system_xn}
    \dot{x} = g_a(x)u^a, \qquad x(0) = x_0
\end{equation}
where $x \in \mathcal{M}$ with $\dim(\mathcal{M})=n$, $u \in \mathcal{U} \subset \mathbb{R}^m$, and $g_a\in\mathcal{T}(\mathcal{M})$ are smooth vector fields.
Consider the fixed time, free endpoint optimal control problem (OCP)
\begin{align} \label{eq:opt_problem}
        \min_{u^a(\cdot)}   &  \quad  J(u^a)  \\
        \text{s.t.}         & \quad\dot{x}^i  = g^i_a(x)u^a \,, \quad x(0) = x_0 \,, \quad \abs{u^a} \leq  u^a_{\text{max}}  \nonumber\,,
 \end{align}
with $t \in \mathcal{I} = [0, T]$ and the cost functional in Bolza form
\begin{align} \label{eq:cost_functional} 
     J(u^a) &= \phi(t,x)\Big|_{t=T} + \int_0^T\psi(t,x,u)\dd t
\end{align} 
where $\phi: \mathcal{I}\times\mathcal{M} \to \mathbb{R}$ and $\psi: \mathcal{I}\times\mathcal{M}\times\mathcal{U} \to \mathbb{R}$ are the terminal and running costs, respectively. 
Applying PMP to the given OCP, as in \cite{bryson1975applied,liberzon_calculus_2012}, yields the first order necessary conditions 
\begin{equation*} \label{eq:two_point_boundary}
    \scalemath{0.95}{\begin{aligned}
    \dot{x}^* &= \partial_\lambda H\of{t,x^*,u^*,\lambda^*}, \quad\quad &x^*\of{0} &= x_0 \\
    \dot{\lambda}^* &= -\partial_x H\of{t,x^*,u^*,\lambda^*}, \quad &\lambda^*\of{T} &= \partial_x\phi\of{T,x^*(T)}\\
\end{aligned}}
\end{equation*}
where $H = \psi(t,x,u) + \lambda_lg^l_a(x)u^a$ denotes the Hamiltonian, $u^*$ the optimal control, $x^*$ the corresponding solution of the initial value problem and $\lambda^*$ the corresponding costate. According to PMP, the optimal control $u^*$ minimizes the Hamiltonian
$ H\of{t,x^*,v,\lambda^*}$ for all $t \in \mathcal{I}$ over the set of admissible control $\mathcal{U}$, i.e. $H\of{t,x^*,v,\lambda^*} \geq H\of{t,x^*,u^*,\lambda^*}\,, \forall\, v\in \mathcal{U}$.\footnote{In the following, arguments and the superscript $^*$ are omitted whenever it is clear from the context.} The characterization of the optimal control $u$ depends on the structure of $H$ with respect to $u$, where in the regular case, the optimal control is determined by the stationary condition $\partial_uH=0$, providing the Hessian $\partial_{uu}^2H$ is regular and $u$ lies in the interior of $\mathcal{U}$. In particular, if $\partial_{uu}^2H\succ 0$ (\emph{strengthened Legendre-Clebsch condition}), then the optimal control is a strict local minimizer of $H$. If the unconstrained optimal $u$ violates the bounds, the constrained optimal control lies on the boundary $\partial\mathcal U$ and the corresponding control component is given by $u^a = \sign(u^a)u^a_{\text{max}}$. A \textit{singular arc} occurs if the Hessian $\partial_{uu}^2H$ is singular for one or more control inputs $u^b$. On a singular arc the stationary condition vanishes identically on a nontrivial interval $\mathcal I_s\subseteq \mathcal I$. Especially, for the $b$-th input, $\xi_b=\partial_{u^b}H\equiv 0\,, \,\, t\in\mathcal I_s$ and therefore the stationary condition does not allow to determine $u^b$ algebraically. To recover $u^b$, one differentiates the stationary condition $\xi_b$ with respect to time until $u^b$ appears explicitly. Thus,
$
    \frac{\dd^k}{\dd t^k}\xi_b\equiv 0,\, k=0,1,\dots,\bar k-1,\,\, t\in\mathcal{I}_s,
$
where $\bar k$ is the smallest integer such that
$
    \frac{\partial}{\partial u^b}\left(\frac{\dd ^{\bar k}}{\dd t^{\bar k}}\xi_b\right)\neq 0 \,,
\, t\in\mathcal{I}_s.
$
Furthermore, the integer $\bar k$ is even and the \emph{order} of the singular arc is defined as $p=\bar k/2$.
Moreover, the \emph{generalized Legendre-Clebsch condition} 
$
(-1)^p \frac{\partial}{\partial u^b}\left(\frac{\dd ^{2p}}{\dd t^{2p}}\xi_b\right)\succeq 0
\quad (\text{or } \succ 0)
$ must hold for $u^b$ to be an optimal candidate \cite{lewis_definitions_1980}. If $H$ is affine in $u^b$, $\xi_b$ is called the switching function and if $\xi_b=\partial_{u^b}H \neq 0$ the optimal control is given by the bounds $\partial \mathcal{U}$. Otherwise, the it follows from the analysis of the singular arc. 

\subsection{Differential Flatness}
 For a system of the form \prettyref{eq:system_xn} with $m=2$ inputs, the problem of flatness has been solved and can be found, e.g. in \cite{martin_feedback_1994, li_describing_2010}.
In the following and throughout this paper, we restrict ourselves to the three dimensional case. We consider $x$-flat systems of the form 
\begin{align} \label{eq:system}
    \dot{x} = g_1(x)u^1 + g_2(x)u^2 \,,
\end{align} where $x \in \mathcal{M}$ with $\dim(\mathcal{M})=3$, and  the input vector fields $g_a\in\mathcal{T}(\mathcal{M})$ form the distribution $\mathcal{D} = \Span{g_1,g_2}$. 
According to \cite[Theorem 5]{li_describing_2010}, $\varphi = (\varphi^1,\varphi^2)$ is a $x$-flat output of \prettyref{eq:system} if an only if $\dd \varphi^1 $ and $\dd \varphi^2$ are locally linearly independent, the annihilator of $\mathcal{L} = (\Span{\dd \varphi^1,\dd \varphi^2})^\perp \subset \mathcal{D}$ and $\mathcal{D} \not\subset \mathcal{L}$. 

 In particular, a system of the form \prettyref{eq:system} is $x$-flat, if and only if, it is equivalent to the so-called chained form 
 \begin{equation}\label{eq:CF}
     \dot{z}^1 = v^1\,, \dot{z}^2 = z^3v^1\,, \dot{z}^3 = v^2,
 \end{equation} 
via a state transformation $z=\Phi_z(x)$ and a static input transformation $v=\Phi_v(x,u)$.
In particular, in these coordinates, the flat output is given by the first two components of the state $\varphi = (z^1,z^2)$. 
For a flat system of the form \prettyref{eq:system}, the time derivative of a flat-output component $\varphi^i(x)$ is given by
\begin{equation*}
    \dot \varphi^i = \lie{g_1}\varphi^iu^1 + \lie{g_2}\varphi^iu^2, \qquad i=1,2\,.
\end{equation*}
Given that $\dot \varphi = (v^1, z^3v^1)$ in the coordinates of \prettyref{eq:CF} and that $\operatorname{rank}\of{\partial_{u}(\dot \varphi^1, \dot \varphi^2)}$ holds regardless of the chosen state and input coordinates, it follows that there always exists a static input transformation $\tilde{u}=\Phi_{\tilde{u}}(x,u)$ such that, in the new input coordinates, the system is given by $\dot{x}=\tilde{g}_1 \tilde{u}^1+\tilde{g}_2 \tilde{u}^2$ and $\lie{\tilde{g}_2}\varphi^i = 0$. Therefore, throughout this work, we assume locally without loss of generality that only the input $u^1$ enters the first time derivative of each flat output component $\varphi^i$ of the considered system \prettyref{eq:system}, meaning  
\begin{equation} \label{eq:lie_g2}
    \lie{g_2}\varphi^i = 0, \qquad i=1,2 \, . 
\end{equation}
Furthermore, for the system in chained form \prettyref{eq:CF}, the vector fields and the corresponding Lie bracket are given by 
\begin{equation*}
 \scalemath{0.95}{
    \begin{aligned}
    g_1 &= \partial_{z^1} + z^3\partial_{z^2}\,, \quad g_2 = \partial_{z^{3}}\,, \quad [g_1,g_2] =   -\partial_{z^{2}} \,.
    \end{aligned}}
\end{equation*} Hence, the Lie derivatives of the flat outputs along these vector fields satisfy 
\begin{equation} \label{eq:lie_deriv}
     \scalemath{0.95}{\begin{aligned}
        \lie{g_1}\varphi^1 &= 1 \\
        \lie{g_1}\varphi^2 &= z^3 \\
    \end{aligned} \qquad
    \begin{aligned}
        \lie{g_2}\varphi^1 &= 0 \\
        \lie{g_2}\varphi^2 &= 0
    \end{aligned}  \qquad
    \begin{aligned}
        \lie{[g_1,g_2]}\varphi^1 &= 0 \\
        \lie{[g_1,g_2]}\varphi^2 &= -1 \,.
    \end{aligned}}
\end{equation}
From \prettyref{eq:lie_deriv} it can be seen that $\lie{g_1}\varphi$ and $\lie{[g_1,g_2]}\varphi$ are locally linearly independent and therefore
 \begin{align} \label{eq:non_deg}
         \scalemath{0.95}{\det\begin{bmatrix}
             \lie{g_1}\varphi^1 & \lie{[g_1,g_2]}\varphi^1 \\
             \lie{g_1}\varphi^2 & \lie{[g_1,g_2]}\varphi^2
         \end{bmatrix} \neq 0}
\end{align}
must hold. The regularity of \prettyref{eq:non_deg} will turn out to be important for our investigations.

\section{Problem Statement}
Although the following proposed formulation differs from the classical LQR/LQT state-space setting, it follows the same underlying design principle, namely a quadratic tracking objective. In our approach, the terminal and running costs are chosen in an LQT-inspired form, but expressed in terms of the tracking error and its time derivative. In the following, the tracking error and its time derivative are defined as
\begin{align*}
    e^i = y_d^i-\varphi^i \,, \qquad
    \dot{e}^i = \dot{y}_d^i-\lie{{g}_1}\varphi^iu^1\,.
\end{align*}
The terminal cost is written as
\begin{equation}\label{eq:terminal_cost}
    \phi(t,x) = \frac{1}{2} e^i\bar{Q}_{ij}e^j \,,
\end{equation} with the smooth desired output trajectory $y_d \in C^\infty(\mathcal{I}, \mathbb{R}^2)$ and the constant positive definite weight matrix $\bar{Q} = [\bar{Q}_{ij}] \succ 0$. Similarly, the running cost is defined as 
\begin{align}\label{eq:running_cost}
    \psi(t,x,u) &=
    \frac{1}{2}\Big(e^iQ_{ij}e^j + \dot{e}^iM_{ij}\dot{e}^j\Big)\,,
\end{align} with the constant positive definite weight matrices $Q = [Q_{ij}], \, M = [M_{ij}] \succ 0$ of the tracking error and its time derivative, respectively. Thus, the Hamiltonian is given by 
\begin{align*}
    H &=    \frac{1}{2}\Big(e^iQ_{ij}e^j + \dot{e}^iM_{ij}\dot{e}^j\Big)
    + \lambda_l g^l_a u^a,
\end{align*}
 and the first order necessary conditions follow as
\begin{align}
    &\quad\,\,\scalemath{0.94}{\dot{\lambda}_k = e^iQ_{ij}\Part{x^k}\varphi^j + \dot{e}^iM_{ij}\Part{x^k}\lie{{g}_1}\varphi^ju^1 - \lambda_l\Part{x^k}g^l_au^a} \label{eq:costate_diff}\\
    &\scalemath{0.94}{\lambda_k\big|_{t=T} =  -e^i\bar{Q}_{ij}\Part{x^k}\varphi^j\big|_{t=T}} \label{eq:boundary_cond} \\
    &\,\scalemath{0.94}{\partial_{u^1}H = -\dot{e}^iM_{ij}\lie{{g}_1}\varphi^j + \lambda_lg^l_1 = \xi_1} \label{eq:opt_u1_stat}\\
    &\,\scalemath{0.94}{\partial_{u^2}H = \lambda_lg^l_2 = \xi_2} \label{eq:opt_u2_stat}\,.
\end{align}
For optimality, the second order necessary condition $\partial^2_{u u}H \succeq 0$ must hold for all $t \in \mathcal{I}$. For the regular control $u^1$, the \textit{strengthened Legendre–Clebsch} $\partial^2_{u^1u^1} H > 0$ is additionally assumed, implying that $u^1$ is a strict local minimizer. Solving the stationary condition \prettyref{eq:opt_u1_stat} for $u^1$, yields
\begin{align} \label{eq:opt_u1}
    u^1 &= \frac{\dot{y}_d^iM_{ij}\lie{g_1}\varphi^j - \lambda_lg^l_1}{A_{11}}
\end{align}
where $A_{11} = \lie{g_1}\varphi^iM_{ij}\lie{g_1}\varphi^j > 0$. In contrast to $u^{1}$, which is directly determined by the stationary condition $\xi_1=0$, the control $u^{2}$ is singular since $\partial^2 _{u^{2}u^{2}}H\equiv0$ due to $\lie{g_2}\varphi^i=0$. Consequently, the stationary condition does not allow for a direct algebraic solution of $u^{2}$. To determine the singular control law, the switching function \prettyref{eq:opt_u2_stat} must be differentiated with respect to time until $u^{2}$ appears explicitly. For such a singular arc to be optimal, the generalized Legendre–Clebsch condition must be satisfied as a necessary second-order condition. Differentiating \prettyref{eq:opt_u2_stat} with respect to time yields
\begin{align} \label{eq:consistency_condition}
       \PartDiffT{}\xi_2 = \Big[-\dot{e}^iM_{ij}\lie{[g_1,g_2]}\varphi^j + \lambda_l[g_1,g_2]^l\Big]u^1 \overset{!}{=} 0\,,
\end{align}
where $u^{2}$ does not appear, since $[g_2,g_2]=0$ trivially. Taking the derivative of \prettyref{eq:consistency_condition} yields
\begin{equation} \label{eq:consistency_condition_order2}
\scalemath{0.96}{\begin{aligned}
        \PartDDiffT{}\xi_2 &= \Big[-(\ddot{y}_d^i - \lie{g_a}\lie{g_1}\varphi^iu^au^1-\lie{g_1}\varphi^i\dot{u}^1)M_{ij}\lie{[g_1,g_2]}\varphi^j \\
        &- \dot{e}^iM_{ij}\lie{\liebrack{g_2}{[g_1,g_2]}}\varphi^ju^2 +  e^iQ_{ij}\lie{[g_1,g_2]}\varphi^j \\
        & - \dot{e}^iM_{ij}\lie{\liebrack{g_1}{[g_1,g_2]}}\varphi^ju^1 + \lambda_l\liebrack{g_a}{[g_1,g_2]}^lu^a\Big]u^1 \\ 
        &+ \underbrace{\Big[-\dot{e}^iM_{ij}\lie{[g_1,g_2]}\varphi^j + \lambda_l[g_1,g_2]^l\Big]}_{\overset{\prettyref{eq:consistency_condition}}{=}0}\dot{u}^1  \overset{!}{=} 0 \, .
\end{aligned}}
\end{equation}
From this equation, $u^{2}$ can be determined, provided that the \textit{strengthened generalized Legendre-Clebsch} condition $-\partial_{u^2}\frac{\dd^{2}}{\dd t^{2}}\xi_2 > 0 ,\, \forall x \in \mathcal{M},\, t \in \mathcal{I}_s$ is satisfied. In the standard PMP formulation, the resulting singular optimal control problem is typically associated with a TPBVP, whose numerical solution is nontrivial due to possible switching behavior and to the differential algebraic structure of the stationary conditions \prettyref{eq:opt_u1_stat} -- \prettyref{eq:consistency_condition_order2}. This difficulty could be avoided by analytically solving the costate differential equation, thereby eliminating the need to solve the TPBVP numerically. This constitutes the main result of the present work.
\section{Main results}
 To avoid the explicit integration of the costate differential equations, we seek a product ansatz for the costate 
\begin{equation} \label{eq:lambda_ansatz}
{\lambda}_k(t,x) = -r_j(t,x) \partial_{x^k} \varphi^j(x)\,.
\end{equation}
For this ansatz to satisfy the optimality conditions and the consistency condition \prettyref{eq:opt_u1_stat}, \prettyref{eq:opt_u2_stat} and \prettyref{eq:consistency_condition} respectively, the geometric implication \prettyref{eq:non_deg} must hold. 
Since the system is flat the distribution $ \Span{g_1, g_2, [g_1,g_2]}$ spans the entire tangent space $\mathcal{T}(\mathcal{M})$, which leads to our first result.
\begin{theorem} \label{th:solution_costate_u}
    Consider the system \prettyref{eq:system} together with the OCP \prettyref{eq:opt_problem}, with the cost functional \prettyref{eq:cost_functional} and the associated terminal and running costs \prettyref{eq:terminal_cost} and \prettyref{eq:running_cost}. The components of the costate $\lambda \in \mathcal{T}^*(\mathcal{M})$ on the singular arc are given by
    \begin{align*}
        \lambda_k = -e^i\bar{Q}_{ij}\Part{x^k}\varphi^j = \dot{e}^iM_{ij}\Part{x^k}\varphi^j \,
    \end{align*}
    with $Q_{ij} = \bar{Q}_{il}M^{lm}\bar{Q}_{mj}$.
\end{theorem}
\begin{proof}
The proof is split into three steps. First, the stationary conditions \prettyref{eq:opt_u1_stat}, \prettyref{eq:opt_u2_stat} and the consistency condition \prettyref{eq:consistency_condition} are used to identify the costate. Second, the ansatz \prettyref{eq:lambda_ansatz} is inserted into the costate differential equation \prettyref{eq:costate_diff} and the boundary condition \prettyref{eq:boundary_cond}. Third, the special choice $r_j=e^i\bar{Q}_{ij}$ is shown to satisfy the resulting relations under the condition 
$ Q_{ij} = \bar{Q}_{il}M^{lm}\bar{Q}_{mj} $. From the stationary conditions \prettyref{eq:opt_u1_stat}, \prettyref{eq:opt_u2_stat} and the consistency condition derived in \prettyref{eq:consistency_condition}, it follows that the costate $\lambda $ must satisfy
\begin{align}
    \lambda_kg^k_2 &= 0 \label{eq:xi2}\\
    (\lambda_k - \dot{e}^iM_{ij}\Part{x^k}\varphi^j)g^k_1 &= 0 \label{eq:xi1}\\
    (\lambda_k - \dot{e}^iM_{ij}\Part{x^k}\varphi^j)[g_1,g_2]^k &= 0 \label{eq:xi2_dot}\,.
\end{align}
Decomposing the costate as $\lambda_k =\dot{e}^iM_{ij}\partial_{x^{k}}\varphi ^{j}+c_k$ and substituting into the optimality conditions, it follows that the residual co-vector $c \in \mathcal{T}^*(\mathcal{M})$ must satisfy $c_k g^k_1  = 0, \,  c_kg^k_2 = 0, \,  c_k [g_1, g_2]^k = 0$.
Since $c$ annihilates the basis of $\mathcal{T}(\mathcal{M})$ it follows that $c=0$ identically and therefore, $\lambda_k =\dot{e}^iM_{ij}\partial_{x^{k}}\varphi ^{j}$. In particular, the costate lies in the span of the differentials of the flat output, i.e. $\lambda \in \Span{\dd \varphi^1, \dd \varphi^2}$, which represents the first part of the proof. 

In the second part, the ansatz \prettyref{eq:lambda_ansatz} is inserted into the costate differential equation \prettyref{eq:costate_diff} and boundary condition \prettyref{eq:boundary_cond} in order to characterize the coefficient $r_j$. To verify that the ansatz satisfies the costate differential equation, \prettyref{eq:lambda_ansatz} is differentiated with respect to time, yielding
\begin{equation*}
    \scalemath{0.92}{
    \begin{aligned}
            \dot{{\lambda}}_k &= -\dot{r}_j\Part{x^k}\varphi^j - r_j\Part{x^l}\Part{x^k}\varphi^jg^l_au^a \\
            &= e^iQ_{ij}\Part{x^k}\varphi^j +  \dot{e}^iM_{ij}\Part{x^k}\lie{{g}_1}\varphi^ju^1 + r_j\Part{x^l}\varphi^j\Part{x^k}g^l_au^a \,.
    \end{aligned}
    }
\end{equation*}
By the use of the identity $\Part{x^l}\varphi^j\Part{x^k}g^l_a  = \Part{x^k}\lie{g_a}\varphi^j - \Part{x^k}\Part{x^l}\varphi^jg^l_a$ and
Schwarz’s Theorem it follows that 
\begin{align} \label{eq:inserted_ansatz}
    \scalemath{0.9}{\Big[e^iQ_{ij} + \dot{r}_j\Big]\Part{x^k}\varphi^j + \big[r_j  + \dot{e}^iM_{ij}\big]\Part{x^k}\lie{g_1}\varphi^ju^1 \overset{!}{=}0}\,,
\end{align} has to be met. To prove that all optimality conditions are satisfied, it must be shown that the boundary condition \prettyref{eq:boundary_cond} holds. Inserting the ansatz \prettyref{eq:lambda_ansatz} into \prettyref{eq:boundary_cond}, gives 
$$-e^i\bar{Q}_{ij}\Part{x^k}\varphi^j\big|_{t=T} = -r_j\Part{x^k}\varphi^j\big|_{t=T}.$$

In the third step, we explicitly determine the function $r_j$.
Inserting \prettyref{eq:lambda_ansatz} into \prettyref{eq:xi1} and \prettyref{eq:xi2_dot} results in 
\begin{equation}\label{eq:opt_inserted_ansatz}
    \begin{aligned} 
    -(r_j + \dot{e}^iM_{ij})\lie{g_1}\varphi^j &= 0 \\
    -(r_j + \dot{e}^iM_{ij})\lie{[g_1,g_2]}\varphi^j &= 0 \,.
\end{aligned}
\end{equation}
Hence, from \prettyref{eq:opt_inserted_ansatz} together with \prettyref{eq:non_deg}, it follows that
\begin{align} \label{eq:r_arc}
    r_j = -\dot{e}^iM_{ij}\,.
\end{align}
Due to \prettyref{eq:r_arc}, the second term in \prettyref{eq:inserted_ansatz} vanishes. Thus, the coefficient in the ansatz \prettyref{eq:lambda_ansatz} must satisfy 
\begin{align}
    \dot{r}_j = -e^iQ_{ij} \,, \quad r_j\big|_{t=T} = e^i\bar{Q}_{ij}\big|_{t=T}\,.
\end{align}
A particular solution is given by
\begin{align} \label{eq:r_solution_ansatz}
    r_l = e^i\bar{Q}_{il} \,,
\end{align} 
which satisfies the terminal condition by construction. Moreover, from \prettyref{eq:r_arc} it follows that 
\begin{align} \label{eq:e_dot}
    \dot{e}^m = -r_lM^{lm} \,,
\end{align} where $M^{lm}$ is the inverse of $M_{ij}$ defined by $\delta^m_i = M_{il}M^{lm}$. Inserting \prettyref{eq:r_solution_ansatz} into \prettyref{eq:e_dot} yields $\dot{e}^m = -e^i\bar{Q}_{il}M^{lm}$.
Hence, the time derivative of \prettyref{eq:r_solution_ansatz} is given by
\begin{align} \label{eq:r_inserted}
    \dot{r}_j &= \dot{e}^m\bar{Q}_{mj} = -e^i\bar{Q}_{il}M^{lm}\bar{Q}_{mj} \,.
\end{align} Inserting \prettyref{eq:r_inserted} together with \prettyref{eq:r_arc} into \prettyref{eq:inserted_ansatz} leads to 
\begin{align*}
       e^i\Big[Q_{ij} - \bar{Q}_{il}M^{lm}\bar{Q}_{mj}\Big]\Part{x^k}\varphi^j \overset{!}{=}0 \,,
\end{align*}
which can only hold true if $Q_{ij} = \bar{Q}_{il}M^{lm}\bar{Q}_{mj}$, or in matrix form $Q = \bar{Q}M^{-1}\bar{Q}$.
By the uniqueness of the solution to this terminal-value problem, it follows 
that~\eqref{eq:lambda_ansatz},
where $r_j=-e^i\bar{Q}_{ij} =  \dot{e}^iM_{ij}$.
\end{proof} 
An immediate consequence of \prettyref{th:solution_costate_u} is that the error dynamic on the singular arc is given by $\dot{e}^j = A^j_ie^i$ with $A^j_i = -\bar{Q}_{il}M^{lj}$. Thus, on the singular arc, the error dynamics is linear and of dimension two. Furthermore, \prettyref{th:solution_costate_u} obviates the need to solve the TPBVP and enables the optimal control input $u^a$ to be implemented in feedback form, as stated in the following theorem. 
\begin{theorem} \label{th:opt_solution}
Consider a system of the form 
\begin{align} \label{eq:system_th}
    \dot{x} = g_1(x)u^1 + g_2(x)u^2 \,,
\end{align}
with $\dim(x)=3$ and the OCP \prettyref{eq:opt_problem}, with the cost functional \prettyref{eq:cost_functional} and the associated terminal and running costs \prettyref{eq:terminal_cost} and \prettyref{eq:running_cost}. An optimal solution candidate is given by 
    \begin{align*}
        u^1 =
        \begin{cases}
            \sign(u^1_{\text{sing}})\,u^1_{\max}, & |u^1_{\text{sing}}| \geq u^1_{\max},\\
            u^1_{\text{sing}}, & \text{otherwise}
        \end{cases}
    \end{align*}
    \begin{align*}
        u^2 =
        \begin{cases}
            \sign(\dot{\xi}_2)\,u^2_{\max}, & \dot{\xi}_2 \neq 0\,,\\
            \sign(u^2_{\text{sing}})\,u^2_{\max} &|u^2_{\text{sing}}| \geq u^2_{\max}\,\text{and}\, \dot{\xi}_2 = 0\\
            u^2_{\text{sing}}, & \text{otherwise,}
        \end{cases}
    \end{align*}
    with
    \begin{equation*}
    \scalemath{0.92}{\begin{aligned}
    u^1_{\text{sing}} &= \frac{\Big[\dot{y}_d^iM_{ij} +  e^i\bar{Q}_{ij}\Big]\lie{g_1}\varphi^j}{A_{11}}  \quad
    u^2_{\text{sing}} = \frac{A_{11}b_2-A_{21}b_1}{(A_{12}A_{21}-A_{11}A_{22})u^1}
    \end{aligned}}
    \end{equation*}
     where
    \begin{equation*}
    \scalemath{0.92}{\begin{aligned}
    A_{11} &= \lie{g_1}\varphi^iM_{ij}\lie{g_1}\varphi^j \qquad \qquad
    A_{12} = \lie{[g_1,g_2]}\varphi^iM_{ij}\lie{g_1}\varphi^j \\ 
    A_{21} &= \lie{g_1}\varphi^iM_{ij}\lie{[g_1,g_2]}\varphi^j \qquad \,\,
    A_{22} = \lie{[g_1,g_2]}\varphi^iM_{ij}\lie{[g_1,g_2]}\varphi^j \\
    b_1 &= ((\ddot{y}_d^i - \lie{g_1}^2\varphi^i(u^1)^2)M_{ij} - e^iQ_{ij})\lie{g_1}\varphi^j\\
    b_2 &= ((\ddot{y}_d^i - \lie{g_1}^2\varphi^i(u^1)^2)M_{ij} - e^iQ_{ij})\lie{[g_1,g_2]}\varphi^j\,
    \end{aligned}}
    \end{equation*}
    and $Q_{ij} = \bar{Q}_{il}M^{lm}\bar{Q}_{mj}$ providing $u^1 \neq 0$, and \prettyref{eq:system_th} is $x$-flat.
\end{theorem} \smallskip

The proof of \prettyref{th:opt_solution} is given in the appendix. Thus, the optimal solution candidate of the OCP is a \textit{bang-singular-bang} structure and on the singular arc, the control law takes the form \prettyref{eq:tracking_law}.

\section{Example} \label{sec:example}
Consider the example of the kinematic model of a steerable axle taken from \cite{Nijmeijer_van_der_Schaft_1990}, where the dynamic model is given by $\dot{x} = g_au^a$ with $g_1 = \sin{x^3}\partial_{x^1} + \cos{x^3}\partial_{x^2}$ and 
$g_2 = \partial_{x^3}$.
The kinematic model admits the flat output $\varphi = (x^1,x^2)$. As a first step, \prettyref{eq:non_deg} has to be verified. For the steerable axle, one obtains $\sin^2 \of{x^3} + \cos^2\of{x^3} = 1 \neq 0$. Moreover, since $g_2 = \partial_{x^3}$, it immediately follows that $\lie{g_2} \varphi^i = 0 \,, \forall i \in \{1,2\} $. For the weighting matrices $Q = \diag(100,100)$ and $M = \diag(1,1)$, the coefficients in \prettyref{th:opt_solution} evaluate to
$A_{11} = 1\,,\, A_{12} = 0\,,\, A_{21} = 0\,,\, A_{22} = 1$. Hence, $A_{11} = 1 > 0\,,\, (A_{11}A_{22}-A_{21}A_{12})(u^1)^2 = (u^1)^2 > 0$ and therefore the corresponding regularity conditions are satisfied for $u^1\neq 0$. The steerable axle provides a simple example in which the theoretical assumptions can be explicitly verified and the resulting control law can be computed in closed form. The closed-loop behavior is illustrated in \prettyref{fig:trajectory}, which shows the resulting trajectories together with the desired trajectory under the control law of \prettyref{th:opt_solution}.

\section{Conclusion}
For the considered system class, the flatness-based characterization provides the geometric structure needed to derive an analytic PMP-based optimal tracking law with explicit singular-arc characterization, thereby avoiding the numerical solution of a TPBVP at each time step. Moreover, it was shown that the singular arc induces a bang-singular-bang structure, while the tracking error dynamics on the singular arc reduces to a linear system of order two. Future work will extend the presented results to the $n$-dimensional case. In addition it is planned to apply the presented method to a broader system class of flat systems that are static feedback equivalent to structurally flat triangular forms, see~\cite{gstottner_structurally_2022} and~\cite{hartl_flat_2025} for the two- and the multi-input case, respectively.
\begin{figure}
  \centering
  \input{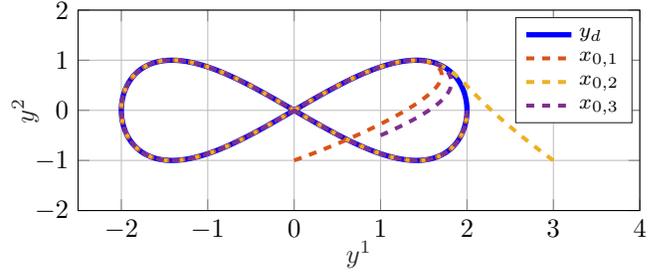}
  \caption{Trajectories for three different initial conditions:
$x_{0,1}=(0,-1,\pi/3)^T$,
$x_{0,2}=(3,-1,-\pi/4)^T$,
$x_{0,3}=(1,-\tfrac{1}{2},\pi/3)^T$.
The solid line denotes the desired trajectory
$y_d(t)=(2\cos(2\pi t/T),\,\sin(\pi t/T))^T$ with $T=5$.
The input bounds are set to $u^1_{\max}=u^2_{\max}=10$.}
  \label{fig:trajectory}
\end{figure}
\section*{APPENDIX}
\begin{proof}[Proof of \prettyref{th:opt_solution}]
    In order to solve \prettyref{eq:opt_problem}, with the cost functional \prettyref{eq:cost_functional} and the associated terminal and running costs \prettyref{eq:terminal_cost} and \prettyref{eq:running_cost}, respectively, \prettyref{th:solution_costate_u} and the following relations are used. Differential flatness for this system class require that $\mathcal{T}(\mathcal{M})=\Span{g_1,g_2,[g_1,g_2]}$ and therefore the higher order lie brackets 
$\liebrack{g_1}{[g_1,g_2]} = \alpha^1g_1 + \alpha^2g_
        2 +\alpha^3[g_1,g_2] ,\,\,
        \liebrack{g_2}{[g_1,g_2]} = \beta^1g_1 + \beta^2g_
        2 +\beta^3[g_1,g_2]$,
can be expressed as linear combinations of the basis vector fields and due to $\lie{g_2}\varphi^i = 0$ it immediately follows that $\lie{[g_1,[g_1,g_2]]}\varphi^j = \alpha^1\lie{g^1}\varphi^j + \alpha^3\lie{[g_1,g_2]}\varphi^j, \,\, 
        \lie{[g_2,[g_1,g_2]]}\varphi^j = \beta^1\lie{g^1}\varphi^j + \beta^3\lie{[g_1,g_2]}\varphi^j$.
By the use of this relation and the conditions \prettyref{eq:opt_u1_stat}, \prettyref{eq:consistency_condition} and
\begin{equation*}
    \scalemath{0.92}{\begin{aligned}
        \lambda_l\liebrack{g_1}{[g_1,g_2]}^l &= \alpha^1\dot{e}^iM_{ij}\lie{g_1}\varphi^j + \alpha^3\dot{e}^iM_{ij}\lie{[g_1,g_2]}\varphi^j \\ 
        \lambda_l\liebrack{g_2}{[g_1,g_2]}^l &= \beta^1\dot{e}^iM_{ij}\lie{g_1}\varphi^j + \beta^3\dot{e}^iM_{ij}\lie{[g_1,g_2]}\varphi^j \,,
\end{aligned}} \end{equation*}
the second order consistency condition \prettyref{eq:consistency_condition_order2} simplifies to 
\begin{equation} \label{eq:consistency_condition_order2_simp}
    \scalemath{0.92}{\begin{aligned}
        \PartDDiffT{}\xi_2 = \Big[&-(\ddot{y}_d^i - \lie{g_a}\lie{g_1}\varphi^iu^au^1-\lie{g_1}\varphi^i\dot{u}^1)M_{ij}\lie{[g_1,g_2]}\varphi^j \\
        &+ e^iQ_{ij}\lie{[g_1,g_2]}\varphi^j\Big]u^1 = 0 \,. 
\end{aligned}} \end{equation}
Next, the stationary condition \prettyref{eq:opt_u1_stat} is differentiated with respect to time, to receive an equation for $\dot{u}^1$, yielding 
\begin{equation} \label{eq:sationary_condition_order1_simp}
    \scalemath{0.92}{\begin{aligned}
        \PartDiffT{}\xi_1 = &-(\ddot{y}_d^i - \lie{g_a}\lie{g_1}\varphi^iu^au^1-\lie{g_1}\varphi^i\dot{u}^1)M_{ij}\lie{g_1}\varphi^j \\
        &+ e^iQ_{ij}\lie{g_1}\varphi^j = 0 \,.
\end{aligned}} \end{equation}
Finally \prettyref{eq:consistency_condition_order2_simp} and \prettyref{eq:sationary_condition_order1_simp} are solved for $\dot{u}^1, u^2$ leading to 
\begin{equation*}
    \scalemath{0.92}{\begin{aligned}
   \dot{u}^1_{\text{sing}} = \frac{A_{12}b_2-A_{22}b_1}{A_{12}A_{21}-A_{11}A_{22}} \quad
    u^2_{\text{sing}} = \frac{A_{11}b_2-A_{21}b_1}{(A_{12}A_{21}-A_{11}A_{22})u^1}
\end{aligned}} \end{equation*} where $A_{11}, A_{22}, A_{12}, A_{21}, b_{1}, b_{2}$ are given in \prettyref{th:opt_solution}.
For $u^1$, the ansatz function \prettyref{eq:lambda_ansatz}  is inserted into \prettyref{eq:opt_u1}, and therefore, the optimal control input $u^1$ on the singular arc is
\begin{equation*}
    \scalemath{0.92}{\begin{aligned}
     u^1_{\text{sing}} &= \frac{\Big[\dot{y}_d^iM_{ij} +  e^i\bar{Q}_{ij}\Big]\lie{g_1}\varphi^j}{A_{11}} \,
\end{aligned}} \end{equation*} 
and therefore 
{\begin{align*}
u^1 =
\begin{cases}
\sign(u^1_{\text{sing}})\,u^1_{\max}, & |u^1_{\text{sing}}| \geq u^1_{\max},\\
u^1_{\text{sing}}, & \text{otherwise.}
\end{cases}
\end{align*}}
 The \emph{Legendre-Clebsch condition} $A_{11} > 0$ and $A_{22} > 0$ hold, if and only if $(\lie{g_1}\varphi^1,\lie{g_1}\varphi^2) \neq (0,0)$ and $(\lie{[g_1,g_2]}\varphi^1,\lie{[g_1,g_2]}\varphi^2) \neq (0,0)$, respectively.  Moreover, by the Cauchy–Schwarz inequality, $A_{21}A_{12} \leq A_{11}A_{22}$ with equality if and only if $\lie{g_1}\varphi$ and  $\lie{[g_1,g_2]}\varphi$ are linearly dependent. Hence, for $u^1\neq0$ the \emph{generalized Legendre-Clebsch condition}  $(A_{11}A_{22}-A_{21}A_{12})(u^1)^2 > 0$ hold, if and only if \prettyref{eq:non_deg} holds. Since, in addition, $\lie{g_2}\varphi^i = 0\,, i \in \{1,2\}$, the flatness conditions of \cite[Theorem 5]{li_describing_2010} are satisfied. Therefore, for the given OCP, the \emph{Legendre-Clebsch condition} together with the \emph{generalized Legendre-Clebsch condition} hold, if and only if the system \prettyref{eq:system} is $x$-flat.
One can easily convince oneself that if the switching function \prettyref{eq:opt_u2_stat} is violated, the optimal control $u^2$ lies on the boundary of $\mathcal{U}$, since the Hamiltonian is affine in $u^2$. Therefore, if $\xi_2 > 0$, one could reduce the value of the Hamiltonian by the choice of the lower bound of $\mathcal{U}$ and vice versa. Howerever, within this work, $\lambda$ is solved analytically and due to $\lie{g_2}\varphi^i = 0$, the switching function  
$\xi_2 \equiv 0$. Therefore, the first consistency condition \prettyref{eq:consistency_condition} determines whether the control is selected from the upper or lower boundary of $\partial \mathcal{U}$. Hence, the time derivative of the switching function $\dot{\xi}_2$ is used to determine if the trajectory leaves the singular manifold given by $\mathcal{S} = \{\xi_2 = \dot{\xi}_2 = 0\}$. In order to show, that the singular manifold is attractive a Lynapunov function of the form $V = \tfrac{1}{2}(\dot{\xi}_2)^2$ is introduced. In contrast to the singular analysis before, it is now assumed that $\dot{\xi}_2 \neq 0$. Therefore, the second time derivative of the switching function, with inserted $\dot{u}^1$,\footnote{$\dot{u}^1$ is calculated from $\dot{\xi}_1$ where $\dot{\xi}_2 \neq 0$.} follows as $\ddot{\xi}_2 = f^1 - f^2u^2$ where 
\begin{equation} \label{eq:xi2_ddot_off_manifold}
    \scalemath{0.85}{\begin{aligned}
        f^1 &= -\frac{A_{11}b_2-A_{21}b_1}{A_{11}}u^1 + \frac{\alpha^3(u^1)^2A_{11}+b_1}{A_{11}u^1}\dot{\xi}_2 \\
        f^2 &= \frac{(A_{11}A_{22}-A_{12}A_{21})(u^1)^2}{A_{11}}-\Big(\beta^3+\frac{A_{12}+A_{21}}{A_{11}} + \frac{\dot{\xi}_2}{(u^1)^2}\Big)\dot{\xi}_2\,.
\end{aligned}} \end{equation} 
By introducing $h^2 = \tfrac{f^1}{f^2}$ it follows, that $\ddot{\xi}_2 = (h^2 - u^2)f^2$. The time derivative of the Lynapunov function is given by $\dot{V} = \dot{\xi}_2\ddot{\xi}_2 = \dot{\xi}_2(h^2 - u^2)f^2$. Introducing the control $u^2 = \sign(\dot{\xi}_2)\,u^2_{\max}$ it follows that $\dot{V} = |\dot{\xi}_2|(\sign(\dot{\xi}_2)h^2 - u^2_{\max})f^2$. Due to the \emph{Legendre-Clebsch condition} $A_{11} > 0$ and \emph{generalized Legendre-Clebsch condition} $(A_{11}A_{22} - A_{12}A_{21})(u^1)^2 > 0 $ it follows from \prettyref{eq:xi2_ddot_off_manifold} that $f^2>\delta>0$  for a sufficient small $|\dot{\xi}_2|$. In the first case, it is assumed that $h^2$ strictly lies in the interior of $\mathcal{U}$. Therefore, one could always find an $\varepsilon > 0$ such that $|h^2| < u^2_{\max} - \varepsilon$. Since $\sign(\dot{\xi}_2)h^2 \leq |h^2| < u^2_{\max} - \varepsilon $ the time deriviative of the Lynapunov can be bounded as  
\begin{align*}
    \dot{V} \leq |\dot{\xi}_2|(u^2_{\max} - \varepsilon - u^2_{\max})f^2 \leq -|\dot{\xi}_2|\varepsilon\delta \leq 0\,.
\end{align*} 
Therefore, in the case that $\dot{\xi}_2\neq0$, the control law is given by $u^2 = \sign(\dot{\xi}_2)\,u^2_{\max}$, which is steering the trajectory toward the manifold $\mathcal{E}=\{\dot{\xi}_2=0\}$. If $|h^2|  > u^2_{\max}$, stability can not be guaranteed. 
For the considered example in \prettyref{sec:example}
it is illustrated in \prettyref{fig:f2_h2}, that in a neighborhood of $\mathcal{S}$, $f^2$ remains strictly positive and $h^2$ lies in the interior of the admissible set $\mathcal{U}$. Thus, the optimal control $u^2$ is given by
\begin{equation*}
    \scalemath{0.92}{\begin{aligned}
        u^2 =
        \begin{cases}
            \sign(\dot{\xi}_2)\,u^2_{\max}, & \dot{\xi}_2 \neq 0\,,\\
            \sign(u^2_{\text{sing}})\,u^2_{\max} &|u^2_{\text{sing}}| \geq u^2_{\max}\,\text{and}\, \dot{\xi}_2 = 0\\
            u^2_{\text{sing}}, & \text{otherwise,}
        \end{cases}
\end{aligned}} \end{equation*} 
The singular expression $u^2_{\mathrm{sing}}$ is valid only on the singular manifold, that is, as long as the consistency condition \prettyref{eq:consistency_condition} is satisfied. If this condition is violated, the singular interior control no longer keeps the trajectory on the singular manifold, and the control must be selected on the boundary of the admissible set $\mathcal{U}$.
\begin{figure}
  \centering
  \input{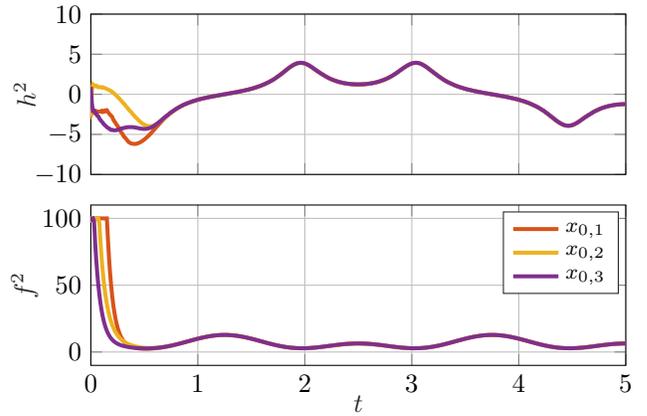}
    \caption{Time evolution of $f^2$ and $h^2=f^1/f^2$ for the trajectories given in \prettyref{fig:trajectory}.
    The plot shows that $f^2$ remains strictly positive in a neighborhood of the
    singular manifold, while $h^2$ stays in the interior of the admissible set
    $\mathcal U$, i.e., $|h^2|<u^2_{\max}$.}
    \label{fig:f2_h2}
\end{figure}
\end{proof}
\bibliographystyle{IEEEtran}
\bibliography{Ref}
\end{document}